\documentclass[11pt]{amsart}

\usepackage[colorlinks,citecolor=blue]{hyperref} 
\usepackage{amsthm} %%% Defines \theoremstyle
\usepackage{amsmath} %%% Defines \numberwithin, \operatorname
\usepackage{amssymb} %%% Defines \mathbb
\usepackage{pinlabel} 
\usepackage{cleveref}

\usepackage{thm-restate} 

\usepackage{tikz}
\usepackage{enumerate} 
\usepackage[numbered]{bookmark} % numbers sections in PDF

\usepackage{todonotes}

\newcommand{\CC}{\mathbb{C}}

\newcommand{\NN}{\mathbb{N}}

\newcommand{\RR}{\mathbb{R}}

\newcommand{\ZZ}{\mathbb{Z}}

\renewcommand{\setminus}{{\smallsetminus}}

\newcommand{\rank}{\operatorname{rank}}

\newcommand{\vv}{{\bf v}}
\newcommand{\ww}{{\bf w}}

\newcommand{\Mod}{\operatorname{Mod}} % Modular group
\newcommand{\GL}{\operatorname{GL}} % General linear
\theoremstyle{plain}
\numberwithin{equation}{section} 
\numberwithin{figure}{section}
\numberwithin{table}{section}
\newtheorem{theorem}{Theorem}[section]
\newtheorem{corollary}[theorem]{Corollary}
\newtheorem{lemma}[theorem]{Lemma}

\newtheorem*{conjecture*}{Conjecture}
\newtheorem{proposition}[theorem]{Proposition}

\theoremstyle{definition}
\newtheorem{definition}[theorem]{Definition}

\newtheorem*{remark*}{Remark}

\newtheorem*{claim*}{Claim}

\newtheorem*{problem*}{Problem}

\newtheorem*{question*}{Question}
\newtheorem*{answer*}{Answer}
\newtheorem*{case*}{Case}
\newtheorem*{application*}{Application}

\crefformat{equation}{(#2#1#3)}
\crefmultiformat{equation}{(#2#1#3)}{ and~(#2#1#3)}{, (#2#1#3)}{ and~(#2#1#3)}
\usepackage{thm-restate}

\theoremstyle{theorem}
\newtheorem*{pennerconst*}{Penner's Construction}

\newcommand{\s}{\mathbf{s}}
\newcommand{\End}{\mathrm{End}}
\newcommand{\e}{\mathbf{e}}
\newcommand\CRG{\mathrm{CRG}}
\newcommand\bG{\mathbf{G}}
\newcommand\crr{\mathrm{cr}}
\newcommand\Rplusmult{\RR_+^\times}
\newcommand\GLp{\mathrm{GL^+}}
\newcommand\Fix{\mathrm{Fix}}
\newcommand\AND{\quad\mbox{and}\quad}
\newcommand\GP{\mathcal{GP}}

\author{Bal\'azs Strenner}
\address{School of Mathematics \\ Georgia Institute of Technology \\
  686 Cherry Street NW, Atlanta GA 30332-0160, USA}
\email{strenner@math.gatech.edu}
\date{\today}
\title[Galois conjugates of pseudo-Anosov stretch factors]{Galois
  conjugates of pseudo-Anosov stretch factors are dense in the complex
plane}

\begin{document}

\begin{abstract}
  In this paper, we study the Galois conjugates of stretch factors of
  pseudo-Anosov elements of the mapping class group of a surface. We
  show that---except in low-complexity cases---these conjugates are
  dense in the complex plane. For this, we use Penner's construction
  of pseudo-Anosov mapping classes. As a consequence, we obtain that
  in a sense there is no restriction on the location of Galois
  conjugates of stretch factors arising from Penner's construction.
  This complements an earlier result of Shin and the author stating
  that Galois conjugates of stretch factors arising from Penner's
  construction may never lie on the unit circle.
\end{abstract}

\maketitle

\section{Introduction}
\label{sec:intro}

Let $S$ be a compact orientable surface. The Nielsen--Thurston
classification theorem \cite{Thurston88} states that every element $f$
of the mapping class group $\Mod(S)$ is either finite order, reducible
or pseudo-Anosov. Associated to every pseudo-Anosov element is a
stretch factor $\lambda >1$ which is an algebraic integer. The goal of
this paper is to study the location of Galois conjugates of
pseudo-Anosov stretch factors in the complex plane.

Let $S_{g,n}$ be the orientable surface of genus $g$ with $n$
boundary components. We define the complexity of $S_{g,n}$ as
$\xi(S_{g,n}) = 3g-3+n$. Note that $\xi(S)$ is half the dimension
of the Teichm\"uller space of $S$.

The main result of the paper is the following.

\begin{theorem}\label{theorem:dense-conjugates-main}
  If $S$ is a compact orientable surface with $\xi(S) \ge 3$, then
  the Galois conjugates of stretch factors of pseudo-Anosov elements
  of $\Mod(S)$ are dense in the complex plane.
\end{theorem}

We proceed with providing motivation for the theorem. Then, at the end of the
introduction, we give an outline of the proof. 

\subsection*{Relation to Fried's problem} Every pseudo-Anosov stretch
factor is a bi-Perron algebraic unit: an algebraic unit $\lambda >1$
whose Galois conjugates other than $\lambda$ and $1/\lambda$ lie in
the annulus $1/\lambda < |z| < \lambda$. Fried \cite{Fried85} asked
whether or not the converse holds (up to taking powers), and it became
a folklore conjecture that it does. This would give a characterization
of the numbers that arise as pseudo-Anosov stretch factors. Assuming
this conjecture, one would expect the Galois conjugates of
pseudo-Anosov stretch factors to be dense in the complex plane.
\Cref{theorem:dense-conjugates-main} is consistent with the
conjecture.

\subsection*{Related results}
Other than the bi-Perron property, little is known about the Galois
conjugates of pseudo-Anosov stretch factors. Nevertheless, stretch
factors of maps appear in many different but related contexts, where
some results about the Galois conjugates are available.

Hamenst\"adt \cite[Theorem 1]{Hamenstadt14} showed that (in an appropriate
sense) typical stretch factors of the homological actions of pseudo-Anosov
mapping classes are totally real.  If the typical pseudo-Anosov
stretch factor was also totally real, then the pseudo-Anosov stretch factors we
construct in this paper would be atypical, since their Galois conjugates
are everywhere in the complex plane.

Thurston \cite{Thurston14} studied the stretch factors of graph maps,
outer automorphisms of free groups and post-critically finite
self-maps of the unit interval. He gave a characterization of such
stretch factors in terms of the location of Galois conjugates: they
are the so-called weak Perron numbers.

Following up on Thurston's work, Tiozzo \cite{Tiozzo13} studied the
fractal defined as the closure of the Galois conjugates of growth
rates of superattracting real quadratic polynomials and showed that
this fractal is path-connected and locally connected. An analogous
fractal for pseudo-Anosov stretch factors would be the closure of the
Galois conjugates of pseudo-Anosov stretch factors $\lambda$
satisfying $\lambda \le T$ for some $T>1$. As far as we know,
this fractal has not yet been studied.

\subsection*{Construction of pseudo-Anosov mapping classes} 
To prove \Cref{theorem:dense-conjugates-main}, we use the following
construction of pseudo-Anosov mapping classes \cite{Penner88} (see
also \cite{Fathi92}).

% Before the theorem, we recall that two simple closed curves on a
% surface are in \emph{minimal position} if they realize the minimal intersection
% number in their homotopy classes, and a collection of simple closed curves $C$ on a
% surface is \emph{filling} if the curves are in pairwise minimal position and
% the components of $S-C$ are disks or once-punctured disks.

\begin{pennerconst*}
  Let $A = \{a_1, \ldots, a_n\}$ and $B = \{b_1, \ldots, b_m\}$ be a pair of
  filling\footnote{The components of the complement of $A$ and $B$ are
    disks or once-punctured disks.} multicurves on an orientable surface $S$.
  Then any product of positive Dehn twists $T_{a_j}$ and negative Dehn twists
  $T_{b_k}^{-1}$ is pseudo-Anosov provided that each curve is used at least
  once.
\end{pennerconst*}

For each pair of multicurves $C = A \cup B$, we denote by $\GP(C)$ the
set of Galois conjugates of stretch factors of pseudo-Anosov elements
of $\Mod(S)$ arising from Penner's construction using the set of
curves $C$. \Cref{theorem:dense-conjugates-main} is a corollary of the
following more concrete statement.

\begin{restatable}{theorem}{maintheorem}\label{theorem:dense_conjugates}
  If $S$ is a compact orientable surface with $\xi(S) \ge 3$, then
  there is a collection of curves $C$ on $S$ such that
  $\overline{\GP(C)} = \CC$.
\end{restatable}

Penner \cite{Penner88} asked if every pseudo-Anosov mapping
class has a power that arises from his construction. This was
answered in the negative by Shin and the author \cite{ShinStrenner15} by showing that
stretch factors arising from Penner's construction do not have Galois
conjugates on the unit circle. In Question 3.1 of the paper
\cite{ShinStrenner15}, we asked if such Galois conjugates can be
arbitrarily close to the unit circle for a fixed collection of curves
$C$. \Cref{theorem:dense_conjugates} answers this question positively.

\subsection*{Low complexity cases}

The hypothesis on the complexity of the surface in
\Cref{theorem:dense-conjugates-main} is necessary, because when
$\xi(S) \le 2$, the Galois conjugates of stretch factors lie on the
real line and the unit circle. This is for the following reasons.

Pseudo-Anosov stretch factors arise as eigenvalues of
integral symplectic matrices of size $2\xi(S)\times 2\xi(S)$. These matrices come
from the integral piecewise linear action of the mapping class on the
measured lamination space of $S$ which has dimension $2\xi(S)$. A
symplectic matrix that has an eigenvalue off the real line and the
unit circle has at least 4 such eigenvalues (by complex conjugation
and the fact that eigenvalues come in reciprocal pairs), so if it also
has a positive real eigenvalue, its size has to be at least
$6 \times 6$. 

Computer experiments suggest that the Galois conjugates are dense in
the real line and the unit circle when $\xi(S) = 2$. However, when
restricting to Penner's construction, the Galois conjugates can only
be positive real. The fact that they cannot lie on the unit circle was
mentioned earlier. The fact that they cannot be negative can be found
in the author's thesis \cite[Section 6.2]{StrennerThesis}.

\subsection*{Sketch of the proof}
We divide the proof into three parts, corresponding to
\Cref{sec:criterion,sec:approximation-of-linear-maps,sec:rich-collections}.

In \Cref{theorem:Galois-conjugate-guarantee}, we give a sufficient
condition for a complex number to be contained in $\GP(C)$ in terms of
the eigenvalues of compositions of certain projections from hyperplanes to
other hyperplanes in $\RR^n$. This reduces to problem of
approximating complex numbers by Galois conjugates of stretch factors
to approximating complex numbers by eigenvalues of compositions of
projections. The proof of this uses results from
\cite{StrennerDegrees} stating that for certain sequences of
pseudo-Anosov mapping classes arising from Penner's construction,
some Galois conjugates of the stretch factors converge, and the
limits are eigenvalues of a composition of projections.

In \Cref{sec:approximation-of-linear-maps}, we define the notion of
\emph{rich} collections of curves, and in
\Cref{theorem:dense-conjugates-criterion} we show that if $C$ is a
rich collection of curves, then $\overline{\GP(C)} = \CC$. The main
ingredient to this is showing that if $C$ is a rich collection, then
every invertible linear transformation of the 2-dimensional plane can
be approximated by compositions of certain projections from
2-dimensional planes to other 2-dimensional planes in $\RR^3$. This
allows us to apply \Cref{theorem:Galois-conjugate-guarantee} to
conclude that all complex numbers are contained in $\overline{\GP(C)}$.

Finally, in \Cref{sec:rich-collections} we construct rich collections
of curves on various surfaces and complete the proof of \Cref{theorem:dense_conjugates}.

\subsection*{Nonorientable surfaces} Penner's construction also works
on nonorientable surfaces \cite{Penner88,StrennerDegrees}, and an
analog of \Cref{theorem:dense_conjugates} could be proven also for
sufficiently complicated nonorientable surfaces. In the orientable
case, we deduce \Cref{theorem:dense_conjugates} as a corollary of
\Cref{theorem:dense-conjugates-criterion} and
\Cref{prop:rich-collections-exist}.
\Cref{theorem:dense-conjugates-criterion} applies to the nonorientable
case as it is, so one would only need to construct rich collections of
curves on nonorientable surfaces.

\section{Galois conjugates of stretch factors in Penner's construction}
\label{sec:criterion}

The goal of this section is to establish a connection between Galois conjugates
of pseudo-Anosov stretch factors and eigenvalues of certain compositions of
projections.

Let $C = \{c_1,\ldots,c_n\}$ be a collection of curves used in
Penner's construction. The intersection matrix $\Omega = i(C,C)$ is the
$n \times n$ matrix whose $(j,k)$-entry is the geometric intersection
number $i(c_j,c_k)$. 

Let $Z_i$ be the orthogonal complement of the $i$th row of $\Omega$.
Since $\Omega$ is an intersection matrix of a collection of filling
curves, all rows are nonzero and the $Z_i$ are hyperplanes. Let
\begin{displaymath}
  p_{i \gets j} : \RR^n \to Z_i
\end{displaymath}
be the---not necessarily orthogonal---projection onto the hyperplane
$Z_i$ in the direction of $\e_j$, the $j$th standard basis vector in
$\RR^n$. This projection is defined if and only if $\e_j$ is not
contained in $Z_i$, which is in turn equivalent to the statement that
the $(i,j)$-entry of $\Omega$ is positive. 

Let $\bG(\Omega)$ be the graph on the vertex set $\{1, \ldots, n\}$
where $i$ and $j$ are connected if the $(i,j)$-entry of $\Omega$ is
positive. For a closed path
\begin{displaymath}
  \gamma = (i_1\cdots i_Ki_1)
\end{displaymath}
in
$\bG(\Omega)$, define the linear map $f_\gamma: Z_{i_1} \to Z_{i_1}$
by the formula
\begin{displaymath}
  f_\gamma = (p_{i_1\gets i_{K}}\circ \cdots
  \circ p_{i_2\gets i_1})|_{Z_{i_1}}.
\end{displaymath}
In words, $f_\gamma$ is a composition of projections: first from
$Z_{i_1}$ to $Z_{i_2}$, then from $Z_{i_2}$ to $Z_{i_3}$, and finally
from $Z_{i_K}$ back to $Z_{i_1}$.

The following theorem gives a sufficient criterion for a complex
number to be approximated by Galois conjugates of stretch factors
arising from Penner's construction using a curve collection $C$.

\begin{theorem}\label{theorem:Galois-conjugate-guarantee}
  Let $C$ be a collection of curves satisfying the hypotheses of
  Penner's construction and let $\Omega = i(C,C)$. Let $\gamma$ be a
  closed path in $\bG(\Omega)$ (not necessarily traversing every vertex). If
  $\theta$ is an eigenvalue of $f_\gamma$ and it is not an algebraic
  unit, then $\theta \in \overline{\GP(C)}$.
\end{theorem}

The main ingredient of the proof is a result from the paper
\cite{StrennerDegrees}. Before we state the theorem, we first recall some
notations from Section 2.3 of that paper.

 Associated to the Dehn twists about the curves $c_i$
are $n\times n$ integral matrices $Q_i$ (depending only on $\Omega$)
with the following property: for a product of the Dehn twists about
the $c_i$ where every twist appears at least once, the corresponding
product $M$ of the $Q_i$ is a Perron--Frobenius matrix whose leading
eigenvalue equals the stretch factor of the pseudo-Anosov map.

The following is the combination of Lemma 1.2 and Theorem 3.1 of
\cite{StrennerDegrees}.

\begin{theorem}\label{theorem:convergence}
  Let $\Omega$ be the intersection matrix of a collection of curves
  satisfying the hypotheses of Penner's construction. Let
  $\gamma = (i_1\ldots i_Ki_1)$ be a closed path in
  $\bG(\Omega)$ visiting each vertex at least once. Let
  \begin{equation}\label{eq:M-gamma}
    M_{\gamma,k} = Q_{i_K}^{k} \cdots Q_{i_1}^{k}
  \end{equation}
  and let $\lambda_k$ be the Perron--Frobenius eigenvalue of
  $M_{\gamma,k}$. Denote by $u_k(x)$ and $v(x)$ the characteristic
  polynomials $\chi(M_{\gamma,k})$ and
  $\chi(f_\gamma)$, respectively. Then we have
  \begin{displaymath}
    \frac{u_k(x)}{x-\lambda_k} \to v(x).
  \end{displaymath}

  If, in addition, $v(\theta) = 0$ and $\theta_k \to \theta$ is a
  sequence such that $u_k(\theta_k) = 0$ and $\theta_k \ne \theta$ for all
  but finitely many $k$, then $\theta_k$ and $\lambda_k$ are
  Galois conjugates for all but finitely many $k$.
\end{theorem}

We are now ready to prove \Cref{theorem:Galois-conjugate-guarantee}.

\begin{proof}[Proof of \Cref{theorem:Galois-conjugate-guarantee}]
  Since $\chi(f_\gamma)$ is invariant under homotopy of $\gamma$
  \cite[Proposition 4.1]{StrennerDegrees} and the graph $\bG(\Omega)$ is
  connected, we may assume that $\gamma = (i_1\ldots i_Ki_1)$ traverses every
  vertex. Then each matrix $M_{\gamma,k}$ corresponds to a pseudo-Anosov
  mapping class with stretch factor $\lambda_k$.

  The characteristic polynomials $u_k(x)$ of $M_{\gamma,k}$ are monic
  and have constant coefficient $\pm 1$, because the matrices $Q_i$
  are invertible. This can be seen directly from the definition of the
  matrices $Q_i$ in Section 2.3 of \cite{StrennerDegrees}. So the roots
  of $u_k(x)$ are algebraic units. 

  By the first part of \Cref{theorem:convergence}, there is sequence
  $\theta_k \to \theta$ such that $u(\theta_k) = 0$ for all $k$. Since
  $\theta$ is assumed not to be an algebraic unit, we have
  $\theta_k \ne \theta$ for all but finitely many $k$. By the second
  part of \Cref{theorem:convergence}, this implies that $\theta_k$ is
  a Galois conjugate of $\lambda_k$. So the number $\theta$ is indeed
  approximated by Galois conjugates of Penner stretch factors arising
  from the collection $C$.
\end{proof}

\section{Approximation of linear maps by compositions of projections}
\label{sec:approximation-of-linear-maps}

In this section we define \emph{rich} collections of curves and prove that if
$C$ is such a collection of curves, then we have $\overline{\GP(C)} = \CC$.
This will reduce our main theorem to the problem of constructing rich
collections of curves on various surfaces. First, we need the following definitions.

We define the \emph{cross-ratio} of a $2\times 2$ matrix
\begin{displaymath}
M =
\begin{pmatrix}
  m_{11} & m_{12} \\
  m_{21} & m_{22} \\
\end{pmatrix}
\end{displaymath}
by the formula $\crr(M) = \frac{m_{11}m_{22}}{m_{12}m_{21}}$. In order
for the cross-ratio to be defined, all matrices are assumed to
have positive entries throughout this section.

Denote by $\Rplusmult$ the multiplicative group of the positive reals.
The \emph{cross-ratio group} $\CRG(M)$ of a matrix $M$ is the subgroup
of $\Rplusmult$ generated by the cross-ratios of all $2\times 2$
submatrices of $M$. Note that any subgroup of $\Rplusmult$ is either
trivial, infinite cyclic or dense.

\begin{definition}
  We call a collection of curves $C = \{c_1,\ldots,c_n\}$ on a surface $S$
  \emph{rich} if
  \begin{itemize}
  \item $C$ fills $S$,
  \item $n\ge 6$ and $C_1 = \{c_1,c_2,c_3\}$ and
    $C_2 = \{c_4,c_5,c_6\}$ form multicurves,
  \item $i(C_1,C_2)$ has positive entries, rank 3 and dense
    cross-ratio group.
  \end{itemize}
\end{definition}

\begin{theorem}[Criterion for density of Galois
  conjugates]\label{theorem:dense-conjugates-criterion}
  If $C$ is a rich collection of curves on $S$, then
  $\overline{\GP(C)} = \CC$.
\end{theorem}

The following subsections develop material necessary for the proof. The proof will
be given at the end of the section.

\subsection{Bipartite subgraphs}
\label{sec:bipartite-subgraph}

Suppose $1\le k_1 < k_2 \le n$ and let
\begin{align*}
I &= \{1,\ldots,k_1\}\\
J &= \{k_1+1,\ldots,k_2\}.
\end{align*}
Suppose that $\omega_{ij} = 0$ whenever
$i,j \in I$ or $i,j \in J$, and $\omega_{ij} > 0$ when
$i \in I,j \in J$ or $j \in I,i\in J$. In other words, we assume that
the subgraph of $\bG(\Omega)$ spanned by the vertices $I \cup J$
is a complete bipartite graph. Compare the setting $I = \{1,2,3\}$ and $J =
\{4,5,6\}$ with the definition of rich collections of curves.

Define the subspace
\begin{displaymath}
V = \langle \e_{k_1+1},\ldots,\e_{k_2}\rangle
\end{displaymath}
generated by the standard basis vectors indexed by $J$. For $1 \le i \le k_1$,
the subspace
\begin{displaymath}
X_i = V \cap Z_i
\end{displaymath}
is a hyperplane in
$V$, because $Z_i$ is hyperplane in $\RR^n$ and
$V \not\subset Z_i$. When $i \in I$ and $j \in J$, the projection
$p_{i \gets j}$ restricts to $V$ and induces a projection
\begin{displaymath}
s_{i \gets j} : V \to X_i
\end{displaymath}
in the direction of $\e_j$. On the other hand, the restriction
$s_{j\gets i}$ of $p_{j \gets i}$ on $V$ is the identity.

Let $\gamma = (i_1j_1\dots i_Kj_Ki_1)$ be a closed path in
$\bG(\Omega)$, starting at $i_1 \in I$, and assume that it only
traverses the vertices in $I \cup J$. Then $f_\gamma$ induces a linear
endomorphism $s_\gamma: X_{i_1} \to X_{i_1}$ by the formula
\begin{displaymath}
  s_\gamma = (s_{i_1\gets j_{K}}\circ s_{j_K\gets i_{K}} \circ \cdots \circ
  s_{i_2\gets j_1} \circ s_{j_1\gets i_1})|_{X_{i_1}},
\end{displaymath}
which simplifies to
\begin{displaymath}
  s_\gamma = (s_{i_1\gets
    j_{K}} \circ \cdots \circ s_{i_3\gets j_2} \circ
  s_{i_2\gets j_1})|_{X_{i_1}},
\end{displaymath}
since the omitted terms are the identity maps. Since $s_\gamma$ is
simply the restriction of $f_\gamma$ to the invariant subspace
$X_{i_1}$, we have the following.
\begin{proposition}\label{prop:s-gamma-f-gamma}
  The characteristic polynomial of $s_\gamma$ divides the
  characteristic polynomial of $f_\gamma$.
\end{proposition}
In other words, the eigenvalues of $s_\gamma$ form a subset of the
eigenvalues of $f_\gamma$. Thus having control over the eigenvalues of
$s_\gamma$ is useful for applying
\Cref{theorem:Galois-conjugate-guarantee}.

In the proof of \cite[Proposition 4.1]{StrennerDegrees}, it was shown that
$f_\gamma$ is invariant under homotopies of $\gamma$ that fix the last edge of
$\gamma$. This property is inherited by $s_\gamma$. In fact a stronger homotopy
invariance holds for $s_\gamma$: it is invariant under \emph{all} homotopies
fixing the base point $i_1$, without the assumption that the last edge of
$\gamma$ is fixed throughout the homotopy. To see this we only need to check
that the removal of the backtracking $i_Kj_Ki_1$, when $i_K=i_1$, from $\gamma$
does not change $s_\gamma$. This is because only the projection
$s_{i_1 \gets j_K}|_{X_{i_K}}$ is dropped from the composition, but it is a
projection from $X_{i_1}$ to $X_{i_1}$ so it does not have any effect.

As a consequence, the map $\gamma \mapsto s_\gamma$ induces a
well-defined map
\begin{displaymath}
\rho_{i_1}: \pi_1(G',i_1) \to \End(X_{i_1})
\end{displaymath}
where
$G'$ is the subgraph of $\bG(\Omega)$ spanned by the vertex set
$I \cup J$ and $\End(X_{i_1})$ is the set of linear endomorphisms of
$X_{i_1}$. Moreover, this map is an anti-homomorphism:
\begin{displaymath}
s_{\gamma_1*\gamma_2} = s_{\gamma_2} \circ s_{\gamma_1}.
\end{displaymath}
This
property reduces the computation of $s_\gamma$ for a long path
$\gamma$ to the computation of $s_\gamma$ for short paths $\gamma$. It
also shows that the image of $\rho_{i_1}$ is in fact in
$\GL(X_{i_1})$.

\subsection{An example}
\label{sec:line-example}

Let $I = {1,2}$ and $J = {3,4}$. The upper left $4\times 4$ submatrix
of $\Omega$ has the block form $
\begin{pmatrix}
  0 & Y \\
  Y^T & 0 \\ 
\end{pmatrix}$, where $Y =
\begin{pmatrix}
  a & b \\
  c & d \\
\end{pmatrix}
$
is a $2\times 2$ matrix with positive entries. Now $V$ is the
2-dimensional subspace generated by $\e_3$ and $\e_4$. The hyperplanes
$Z_1$ and $Z_2$ are the orthogonal complements of first and second
rows of $\Omega$. Hence $X_1$ and $X_2$ are the lines in $V$ with
equations $a\e_3 + b\e_4 = 0$ and $c\e_3 + d\e_4 = 0$. The slopes are
$-\frac{b}a$ and $-\frac{d}c$, respectively. The lines are illustrated on \Cref{fig:two_lines_proj}.

  \begin{figure}[ht]
    \centering
    \begin{tikzpicture}[scale=0.9]
      \newcommand\commonbase[5]{
        \draw[->] (-2,0) -- (4,0);
        \draw[->] (0,-4) -- (0,2); 
        \draw[step=1cm,very thin,densely dotted] (-1.5,1.5) grid (3.5,-3.5);
        \node[] at (2.5,1.6) {#1};
        \draw[very thick,->] (0,0) -- (1,0) node[above] {#2};
        \draw[very thick,->] (0,0) -- (0,1) node[right] {#3};
        \draw (-1.5,1.5) -- (3.5,-3.5) node[right,] {#4};
        \draw (-1.5,1) -- (3.5,-2.33) node[right,] {#5};
      }

      \begin{scope}
        \commonbase{$V$}{$\mathbf{e}_3$}{$\mathbf{e}_4$}{$X_1$}{$X_2$};
        \draw[thick,->] (2,-2) -- (3,-2);
        \draw[thick,->] (3,-2) -- (3,-3);
        \draw[thick,->,yshift=-3pt,xshift=-3pt] (2,-2) -- (3,-3);
      \end{scope}

    \end{tikzpicture}

    \caption{The map $s_\gamma\colon X_1 \to X_1$.}
    \label{fig:two_lines_proj}
  \end{figure}

  Let $\gamma$ be some non-contractible closed path of length 4 in the
  subgraph $G'$ of $\bG(\Omega)$ spanned by $I \cup J$. For instance, let
  $\gamma = (14231)$. Then
  $s_\gamma = s_{1 \gets 3}|_{X_2} \circ s_{2\gets 4}|_{X_1}$. The
  projection $s_{2\gets 4}|_{X_1} : X_1 \to X_2$ changes only the
  $\e_4$-coordinate of points, and it changes it by a factor
  $\frac{ad}{bc}$ which is the ratio of the slopes of $X_2$ and $X_1$.
  The projection $s_{1 \gets 3}|_{X_2}$ then projects back onto $X_1$
  without changing the $\e_4$-coordinate. Hence the composition is a
  scaling of the line $X_1$ by a factor $\frac{ad}{bc}$. Therefore
  $\frac{ad}{bc}$ is an eigenvalue of $s_\gamma$ and $f_\gamma$.

The graph $G'$ is topologically a circle and
$\rho_1: \pi_1(G',1) \to \End(X_1)$ maps the generator of the infinite
cyclic group to the scaling of $X_1$ by $\frac{ad}{bc}$. As a consequence,
the eigenvalues of $s_\gamma$ for closed paths $\gamma$ in $G'$ are
precisely the integer powers of $\frac{ad}{bc}$.

% The quantity $\frac{ad}{bc}$ associated to the $2\times 2$ matrix $Y$
% will play an important role later on. We call it the \emph{cross
%   ratio} of $Y$.

\subsection{Short closed paths}
\label{sec:short_closed_paths}

The content of this section is a geometric description of $s_\gamma$
when $\gamma$ is a path of length 4. \Cref{sec:line-example}
discussed the simple special case when $G'$ is a graph on four
vertices. In this section we allow $G'$ to be bigger, hence $V$ to
have dimension greater than 2.

Consider a closed path $(i_1j_1i_2j_2i_1)$ where
$i_1,i_2\in I$ and $j_1,j_2 \in J$. Associated to $(i_1j_1i_2j_2i_1)$
are the following data:
\begin{itemize}
\item $\Omega_{i_1i_2}^{j_1j_2} =
  \begin{pmatrix}
    \omega_{i_1j_1} & \omega_{i_1j_2} \\
    \omega_{i_2j_1} & \omega_{i_2j_2} \\
  \end{pmatrix}$,
\item $c_{i_1i_2}^{j_1j_2} = \crr(\Omega_{i_1i_2}^{j_1j_2})$, and
\item
  $s_{i_1i_2}^{j_1j_2} = s_{(i_1j_1i_2j_2i_1)} = s_{i_1 \gets
    j_2}|_{X_{i_2}} \circ s_{i_2 \gets j_1}|_{X_{i_1}}\in
  \GL(X_{i_1})$.
\end{itemize}

Note that $\big(s_{i_1i_2}^{j_1j_2}\big)^{-1} = s_{i_1i_2}^{j_2j_1}$,
since $(i_1j_1i_2j_2i_1)$ and $(i_1j_2i_2j_1i_1)$ represent inverse
elements in $\pi_1(G',i_1)$. So the monoid generated by the linear
maps $s_{i_1i_2}^{j_1j_2} \in \GL(X_{i_1})$, where
$(i_1j_1i_2j_2i_1)$ runs through all closed paths of length 4 in
$G'$ with base point $i_1$, is actually a group.

\begin{proposition}\label{prop:simple_composition}
  Suppose $(i_1j_1i_2j_2i_1)$ is a homotopically nontrivial closed
  path in $G'$. Then
  \begin{enumerate}[(i)]
  \item\label{item:identity_action} $s_{i_1i_2}^{j_1j_2}$ acts as the
    identity on $F = X_{i_1} \cap X_{i_2}$;
  \item\label{item:intersection_line} The 2-dimensional subspace
    $\langle \e_{j_1}, \e_{j_2} \rangle$ is not contained in the
    hyperplane $X_{i_1}$, therefore
    $L = X_{i_1} \cap \langle \e_{j_1}, \e_{j_2} \rangle$ is a line;
  \item\label{item:stretch_line} $s_{i_1i_2}^{j_1j_2}$ stretches the
    line $L$ by a factor of $c_{i_1i_2}^{j_1j_2}$.
  \end{enumerate}
  Moreover, if $c_{i_1i_2}^{j_1j_2} \ne 1$, then $F$ is a codimension
  1 subspace in $X_{i_1}$, and we have $X_{i_1} = F \oplus L$.
\end{proposition}
\begin{proof}
  We first prove statement (\ref{item:identity_action}). The map
  $s_{i_1i_2}^{j_1j_2}$ is the composition of a projection from
  $X_{i_2}$ to $X_{i_1}$ and a projection from $X_{i_1}$ to $X_{i_2}$,
  both of which act on $X_{i_1} \cap X_{i_2}$ as the identity.

  For part (\ref{item:intersection_line}), note that $j_1\ne j_2$, otherwise
  $(i_1j_1i_2j_2i_1)$ is contractible. Hence
  $\langle \e_{j_1}, \e_{j_2}\rangle$ is a 2-dimensional subspace.
  Since $\omega_{i_1j_1} \ne 0$ and $\omega_{i_1j_2} \ne 0$, the
  vectors $\e_{j_1}$ and $\e_{j_2}$ are not orthogonal to the row
  vector $\e_{i_1}^T \Omega$, hence they are not contained in
  $X_{i_1}$.

  For part (\ref{item:stretch_line}), observe that the line
  $X_{i_1} \cap \langle \e_{j_1}, \e_{j_2} \rangle$ is generated by
  $\vv = \omega_{i_1j_2} \e_{j_1} - \omega_{i_1j_1} \e_{j_2}$. Since
  $s_{i_2 \gets j_1}$ is a projection on $X_{i_2}$ in the direction of
  $\e_{j_1}$, we have
  \begin{displaymath}
    s_{i_2 \gets j_1} (\vv) = \frac{\omega_{i_1j_1}}{\omega_{i_2j_1}}(\omega_{i_2j_2} \e_{j_1} - \omega_{i_2j_1} \e_{j_2}).
  \end{displaymath}
  Similarly,
  \begin{displaymath}
    s_{i_1 \gets j_2}(s_{i_2 \gets j_1} (\vv)) = \frac{\omega_{i_1j_1}\omega_{i_2j_2}}{\omega_{i_2j_1}\omega_{i_1j_2}}(\omega_{i_1j_2} \e_{j_1} - \omega_{i_1j_1} \e_{j_2}) =
    c_{i_1i_2}^{j_1j_2} \vv.
  \end{displaymath}

  Finally, the condition $c_{i_1i_2}^{j_1j_2} \ne 1$ implies that the rows of
  $\Omega_{i_1i_2}^{j_1j_2}$ are not constant
  multiples of each other, hence $X_{i_1} \ne X_{i_2}$. Since both
  $X_{i_1}$ and $X_{i_2}$ are hyperplanes in $V$, their intersection
  has codimension 1 in $X_{i_1}$. 
\end{proof}

In summary, the linear map $s_\gamma$ takes a very simple form when
$\gamma = (i_1j_1i_2j_2i_1)$ is a closed path with
$c_{i_1i_2}^{j_1j_2} \ne 1$: it fixes a hyperplane and stretches a
line by the positive factor $c_{i_1i_2}^{j_1j_2}$. In the next section
we consider these linear maps as building blocks for constructing
more complicated linear maps.

\subsection{Linear endomorphisms of planes}

Let $W$ be a 2-dimensional vector space over $\RR$. For any
$\vv \in W$ and $a > 0$ denote by 
\begin{itemize}
\item $\Fix(\vv,a) \subset \GLp(W)$ the group of linear maps fixing
  the vector $\vv$ and having determinant $a^n$ for some integer $n$;
\item $\Fix(\vv) \subset \GLp(W)$ the group of linear maps
  fixing the vector $\vv$ and having positive determinant.
\end{itemize}

\begin{lemma}\label{lemma:halfplane}
  Let $\vv\in W$ and $f_1,f_2\in \Fix(\vv)$. Suppose that $f_1$ and
  $f_2$ have linearly independent eigenvectors $\ww_1$ and $\ww_2$ with
  eigenvalues $a_1 >1$ and $a_2 >1$, respectively. Then $\Fix(\vv,a_1)$
  and $\Fix(\vv,a_2)$ are contained in the closure
  $\overline{\langle f_1, f_2 \rangle} \subset \GLp(W)$.
\end{lemma}
\begin{proof}
  In the basis $(\ww_1,\vv)$ the maps $f_1$ and $f_2$ are described by
  the matrices
  \begin{displaymath}
  A_1 =
  \begin{pmatrix}
    a_1 & 0 \\
    0 & 1 \\
  \end{pmatrix}
  \quad
  \mbox{and} 
  \quad
  A_2 =
  \begin{pmatrix}
    a_2 & 0 \\
    b & 1 \\
  \end{pmatrix}
\end{displaymath}
where $b\ne 0$. Define
  \begin{displaymath}
    C_n = A_1^n A_2 A_1^{-n} =
    \begin{pmatrix}
      a_2 & 0 \\
      \frac{b}{a_1^n} & 1 \\
    \end{pmatrix}
  \end{displaymath}
  and let 
  \begin{displaymath}
    C_{n+1}C_n^{-1} =
    \begin{pmatrix}
      1 & 0 \\
      \frac{b(1-a_1)}{a_1^{n+1}a_2} & 1 \\
    \end{pmatrix}.
  \end{displaymath}
  The bottom left entry of $C_{n+1}C_n^{-1}$ tends to zero, therefore
  powers of $C_{n+1}C_n^{-1}$ are dense in the subgroup of matrices of
  the form
  \begin{displaymath}
    \begin{pmatrix}
      1 & 0 \\
      * & 1 \\
    \end{pmatrix}.
  \end{displaymath}

  Finally, note that multiplying these matrices by powers of $A_i$
  yields everything in
  \begin{displaymath}
    \Fix(\vv,a_i) =\left\{
      \begin{pmatrix}
        a_i^n & 0 \\
        t & 1 \\
      \end{pmatrix}
      : n \in \ZZ, t\in \RR
    \right\}
  \end{displaymath}
  for $i = 1,2$.
\end{proof}

\begin{lemma}\label{lemma:two_halfplanes}
  Let $\vv_1,\vv_2 \in W$ be linearly independent vectors, and let
  $a_1,a_2 >0$ such that $\langle a_1,a_2 \rangle$ is dense in
  $\Rplusmult$. Then
  \begin{displaymath}
    \overline{\langle \Fix(\vv_1,a_1),\Fix(\vv_2,a_2) \rangle} =
    \GLp(W).
\end{displaymath}

\end{lemma}
\begin{proof}
  Under the action of $\Fix(\vv_1,a_1)$, the orbit of any vector that
  is not a constant multiple of $\vv_1$ is a collection of lines
  parallel to $\vv_1$. A similar statement is true for
  $\Fix(\vv_2,a_2)$. Since $\vv_1$ and $\vv_2$ are linearly independent,
  it follows that
  \begin{displaymath}
    H = \langle \Fix(\vv_1,a_1),\Fix(\vv_2,a_2) \rangle
  \end{displaymath}
  acts transitively on nonzero vectors.

  Conjugating $\Fix(\vv_1,a_1)$ by an element of $H$ mapping $\vv_1$
  to $\vv_2$ shows that $\Fix(\vv_2,a_1) \subset H$. Since
  $\langle a_1,a_2 \rangle$ is dense in $\Rplusmult$, we have
  \begin{displaymath}
    \overline {\langle \Fix(\vv_2,a_1), \Fix(\vv_2,a_2) \rangle} = \Fix(\vv_2)
  \end{displaymath}
  and
  $\Fix(\vv_2) \subset \overline{H}$. 

  Finally, notice that any $f\in \GLp(W)$ can be written as
  $f_2\circ f_1$ where $f_1 \in H$ sends $\vv_2$ to $f(\vv_2)$ and
  $f_2 \in \Fix(\vv_2)$. Hence $\overline{H} = \GLp(W)$.
\end{proof}

\subsection{Cross-ratio groups of matrices}

The final ingredient for the proof of \Cref{theorem:dense-conjugates-criterion}
is \Cref{lemma:quadratic_factors} below, which relates dense cross-ratio groups to the
density of eigenvalues of the maps $f_\gamma$ in the complex plane.

\begin{lemma}\label{lemma:crr_identity}
  If $M$ is a $2\times 3$ matrix and $M_i$ denotes its $2\times
  2$ submatrix obtained by deleting the $i$th column, then
  $\crr(M_1)\crr(M_3)  = \crr(M_2).$
\end{lemma}
\begin{proof}
  % If
  % \begin{displaymath}
  %   M =
  %   \begin{pmatrix}
  %     m_{11} & m_{12} & m_{13} \\
  %     m_{21} & m_{22} & m_{23} \\
  %   \end{pmatrix},
  % \end{displaymath}
  % We have
  $\crr(M_1)\crr(M_3) = \frac{m_{12}m_{23}}{m_{22}m_{13}}\cdot
    \frac{m_{11}m_{22}}{m_{12}m_{21}} = \frac{m_{11}m_{23}}{m_{13}m_{21}} = \crr(M_2).$
\end{proof}

\begin{corollary}\label{cor:crr_in_a_row}
  If $M$ is a $2\times 3$ matrix with nontrivial cross-ratio group,
  then the cross-ratio of at least two of its $2\times 2$ submatrices
  is not 1.
\end{corollary}        

\begin{proposition}\label{prop:2by2_submatrices}
  Let $M$ be a $3\times 3$ matrix of full rank and dense cross-ratio
  group. Then $M$ has two $2\times 2$ submatrices with the following
  two properties:
  \begin{enumerate}[(i)]
  \item\label{item:notadjacent} they are not contained in the same two
    rows or the same two columns
  \item\label{item:dense} their cross-ratios generate a dense subgroup
    of $\Rplusmult$.
  \end{enumerate}
\end{proposition}
\begin{proof}
  Since $M$ has dense cross-ratio group, there are two $2\times 2$
  submatrices $M_1$ and $M_2$ that satisfy (\ref{item:dense}). If they
  also satisfy (\ref{item:notadjacent}), then we are done, so assume
  for example that they are in the same two columns. Then we can
  replace $M_2$ by another $2\times 2$ matrix in the same two rows so
  that (\ref{item:dense}) is still satisfied; otherwise the other two
  $2\times 2$ submatrices sharing the rows of $M_2$ would have
  cross-ratios that are rational powers of $\crr(M_1)$, and by
  \Cref{lemma:crr_identity} the same would be true for $\crr(M_2)$.
\end{proof}

\begin{lemma}\label{lemma:quadratic_factors}
  Suppose that the upper left $6\times 6$ submatrix of $\Omega$ has
  the form
  \begin{displaymath}
  \begin{pmatrix}
    0 & Y \\
    Y^T & 0 \\
  \end{pmatrix}
\end{displaymath}
  for a $3\times 3$ matrix $Y$ with positive entries such that
  $\rank(Y) = 3$ and $\CRG(Y)$ is dense in $\Rplusmult$.

  Let $\varepsilon > 0$ and $u(x) = x^2 + ax + b \in \RR[x]$ be
  arbitrary where $b>0$. Then there exists a closed path $\gamma$ in
  $\bG(\Omega)$, and $v(x) = x^2 + a'x + b'\in\RR[x]$ with
  $|a-a'| < \varepsilon$ and $|b-b'| < \varepsilon$ such that
  $v(x)\, |\, \chi(f_\gamma)$.
\end{lemma}
\begin{proof}
  We assume the notations of
  \Cref{sec:bipartite-subgraph,sec:short_closed_paths} with
  $I = \{1,2,3\}$ and $J = \{4,5,6\}$. Note that $X_1$ is a
  2-dimensional subspace in $V$. By \Cref{prop:2by2_submatrices}, we
  may assume without loss of generality that $c_{12}^{45}$ and
  $c_{13}^{46}$ are not rational powers of each other.

  For all $i \in \{2,3\}$ and $j,j' \in \{4,5,6\}$ the map
  $s_{1i}^{jj'} \in \GLp(X_{1})$ acts as the identity on the line
  $F_i = X_{i} \cap X_1$ and stretches the line
  $L_{jj'} = X_{1}\cap \langle \e_{j}, \e_{j'} \rangle$ by
  $c_{1i}^{jj'}$ by \Cref{prop:simple_composition}. Note that the fact
  that $Y$ has full rank implies that $F_{2}$ and $F_{3}$ are distinct
  and $L_{45}$, $L_{56}$ and $L_{46}$ are pairwise distinct.

  Let $H$ be the submonoid of $\GLp(X_1)$ generated by the maps
  $s_{1i}^{jj'}\in \GLp(X_1)$ where $i \in \{2,3\}$ and
  $j,j' \in \{4,5,6\}$. It is in fact a subgroup, since
  $s_{1i}^{jj'}$ and $s_{1i}^{j'j}$ are inverses of each other
  (cf.~\Cref{sec:short_closed_paths}).

  The maps $s_{12}^{45}$, $s_{12}^{56}$ and $s_{12}^{46}$ are elements
  of $\GLp(X_1)$ fixing the line $F_{2}$ and stretching along the
  lines $L_{45}$, $L_{56}$ and $L_{46}$, respectively. Recall that the
  indices were chosen so that $c_{12}^{45} \ne 1$, so by
  \Cref{cor:crr_in_a_row}, at least two of the stretch factors
  $c_{12}^{45}$, $c_{12}^{56}$ and $c_{12}^{46}$ are different from 1.
  By \Cref{lemma:halfplane}, $\Fix(F_{2},c_{12}^{45})$ is contained in
  $\overline{H}$. By applying the same reasoning for the maps
  $s_{13}^{45}$, $s_{13}^{56}$ and $s_{13}^{46}$, we get that
  $\Fix(F_{3},c_{23}^{56})$ is also contained in $\overline{H}$.
  \Cref{lemma:two_halfplanes} then implies $\overline{H} = \GLp(X_1).$

  Now pick an element $g \in \GLp(X_1)$ with $\chi(g) = u(x)$. Then pick
  an $h\in H$ sufficiently close to $g$ so that
  $\chi(h) = x^2 + a'x + b'$ satisfies $|a-a'| < \varepsilon$ and
  $|b-b'| < \varepsilon$. By the definition of $H$, there is a closed
  path $\gamma$ with base point $1$ visiting only the vertices $1,2,3$
  and $4,5,6$ such that $h = s_\gamma$. So the statement follows by
  \Cref{prop:s-gamma-f-gamma}.
\end{proof}

\begin{proof}[Proof of \Cref{theorem:dense-conjugates-criterion}]
  Let $\Omega = i(C,C)$ and let $\theta\in \CC \setminus \{0\}$ be
  arbitrary. Then $v(x) = (x-\theta)(x-\bar\theta) \in \RR[x]$ has
  positive constant term, so by \Cref{lemma:quadratic_factors} there
  exist
  \begin{itemize}
  \item a sequence $(v_j(x))_{j\in \NN}$ of monic quadratic polynomials
    in $\RR[x]$ and
  \item a sequence $(\gamma_j)_{j\in \NN}$ of closed paths in
    $\bG(\Omega)$
  \end{itemize}
  such that
  \begin{itemize}
  \item $v_j(x) \to v(x)$ and
  \item $v_j(x)| \chi(f_{\gamma_j})$ for every
    $j\in \NN$.
  \end{itemize}
  Let $(\theta_j)_{j\in \NN}$ be a sequence such that
  $v_j(\theta_j) = 0$ for all $j\in \NN$ and $\theta_j \to \theta$.

  Assume for a moment that $\theta$ is not an algebraic unit of degree at most
  2, that is, $\theta$ is not a root of a polynomial $x^2+sx\pm 1$ for some
  $s\in \ZZ$. In this case, $\theta_j$ is not an algebraic unit if $j$ is large
  enough, because the set of algebraic units of degree at most 2 is a discrete
  subset of $\CC$. Hence by \Cref{theorem:Galois-conjugate-guarantee}, we have
  $\theta \in \overline{\GP(C)}$.

  To complete the proof, note that the set of $\theta \in \CC$ where
  $\theta \ne 0$ and $\theta$ is not an algebraic unit of degree at most 2 is
  dense in $\CC$ so we have $\overline{\GP(C)} = \CC$.
\end{proof}

\section{Construction of rich collection of curves}
\label{sec:rich-collections}

In this section, we show that rich collections of curves exist on
sufficiently complicated surfaces.

\begin{proposition}\label{prop:rich-collections-exist}
  If $\xi(S) \ge 3$, then there exists a rich collection of
  curves on $S$.
\end{proposition}

First we prove a lemma about intersection matrices. For a multicurve
$B = \{b_1,\ldots,b_\ell\}$ on $S$ and a vector
$\s = (s_1,\ldots,s_\ell)$ with integer coordinates, the product
$T_B^\s = \prod_{j=1}^\ell T_{b_j}^{s_j}$ is called a
\emph{multitwist} about the multicurve $B$.

\begin{lemma}\label{lemma:two_multicurve_construction}
  Let $A$ and $B$ be multicurves of $S$. If $\s>0$ or $\s<0$, then
  \begin{displaymath}
    i(A,T_B^{\s}(A)) = i(A,B) D_{|\s|} i(B,A)
  \end{displaymath}
  where $D_{|\s|}$ is the $\ell\times \ell$ diagonal matrix with
  entries $|s_1|,\ldots,|s_\ell|$ on the diagonal.
\end{lemma}
\begin{proof}
  If $A,B$ and $C$ are multicurves on $S$ and $\s > 0$ or $\s<0$, then we have
  \begin{displaymath}
    \left|i(a,T_B^{\s}(c)) - \sum_{j=1}^\ell|s_j| i(a,b_j)i(b_j,c)\right| \le i(a,c)
  \end{displaymath}
  for all $a\in A$ and $c\in C$ \cite[Prop.~3.4]{FarbMargalit12}. We
  can summarize these inequalities in the single inequality
  \begin{equation}\label{eq:intersections}
    \left|i(A,T_B^{\s}(C)) - i(A,B) D_{|\s|} i(B,C) \right| \le i(A,C).
  \end{equation}
  We obtain the claimed equation by setting $C=A$.
\end{proof}

\begin{proof}[Proof of \Cref{prop:rich-collections-exist}]
  Consider the pairs of multicurves $A$, $B$ on $S_{0,6}$, $S_{1,3}$
  and $S_{2,0}$, pictured on \Cref{fig:curves}.
  \begin{figure}[ht]
    \centering
    \includegraphics{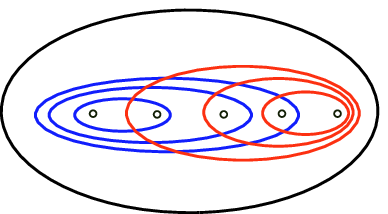}
    \includegraphics{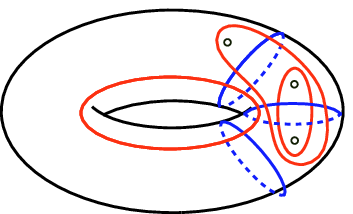}
    \includegraphics{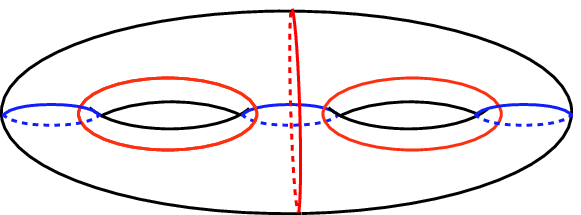}
    \caption{Curves}
    \label{fig:curves}
  \end{figure}
  In the three cases, the intersection matrix $i(A,B)$ is
  \begin{displaymath}
    \begin{pmatrix}
      2 & 2 & 2 \\
      0 & 2 & 2 \\
      0 & 0 & 2 \\
    \end{pmatrix}, \quad
    \begin{pmatrix}
      1 & 1 & 1 \\
      0 & 2 & 2 \\
      0 & 0 & 2 \\      
    \end{pmatrix} \AND
    \begin{pmatrix}
      1 & 1 & 0 \\
      0 & 2 & 0 \\
      0 & 1 & 1 \\
    \end{pmatrix},
  \end{displaymath}
  respectively. In the first two cases $i(A,B)i(B,A)$ is
  \begin{displaymath}
    \begin{pmatrix}
      12 & 8 & 4 \\
      8 & 8 & 4 \\
      4 & 4 & 4 \\
    \end{pmatrix} \AND
    \begin{pmatrix}
      3 & 4 & 2 \\
      4 & 8 & 4 \\
      2 & 4 & 4
    \end{pmatrix}.
  \end{displaymath}
  In the third case
  \begin{displaymath}
    i(A,B)
    \begin{pmatrix}
      1 & 0 & 0  \\
      0 & 1 & 0  \\
      0 & 0 & 2  \\
    \end{pmatrix}
    i(B,A) =
    \begin{pmatrix}
      2 & 2 & 1 \\
      2 & 4 & 2 \\
      1 & 2 & 3
    \end{pmatrix}.
  \end{displaymath}
  By \Cref{lemma:two_multicurve_construction}, we obtain a pair of
  multicurves on all three surfaces with rank 3 intersection matrix
  with positive entries and dense cross-ratio group. The curves
  necessarily fill the surface in each case, so we obtain a
  rich collection of curves.

  Any other compact orientable surface with $\xi(S) \ge 3$
  can be obtained from the three surfaces above by removing open disks
  and taking connected sums with tori. Hence we obtain a collection
  $C$ on all these surfaces that satisfy all properties of richness
  except the filling property. However, the filling property is easily
  achieved by extending both multicurves to maximal multicurves, being
  careful not to include the same curve in both maximal multicurves.
\end{proof}

We are now ready to prove \Cref{theorem:dense_conjugates}.

\maintheorem*

\begin{proof}
  There exists a rich collection of curves by
  \Cref{prop:rich-collections-exist}. By
  \Cref{theorem:dense-conjugates-criterion}, this implies that
  $\overline{\GP(C)} = \CC$.
\end{proof}

\subsection*{Acknowledgements}
We are grateful to Ursula Hamenst\"adt, Autumn Kent, Dan Margalit and
the referees for their comments and help.

\bibliographystyle{alpha}
\bibliography{../mybibfile}

\end{document}